\documentclass[11pt]{article}

\usepackage{fancyhdr}
\usepackage{epsfig}
\usepackage{color}

\usepackage{graphicx}
\usepackage{latexsym}
\usepackage{amssymb}
\usepackage{amsfonts}
\usepackage{amsmath}
\usepackage{amsthm}

\theoremstyle{plain}
\newtheorem{thm}{Theorem}
\newtheorem{lem}{Lemma}
\newtheorem*{thmBG}{Theorem BG}

\DeclareMathOperator{\Ex}{{\bf E}}
\DeclareMathOperator{\Var}{{\bf Var}}

\def\boldhline{\noalign{\global\arrayrulewidth.8pt}\hline\noalign{\global\arrayrulewidth.4pt}}

\newcommand{\iid}{\mbox{i.i.d.\frenchspacing}}
\newcommand{\al}{\mbox{al.\frenchspacing}}

\begin{document}

\LARGE
\begin{center}
\textbf {An Edgeworth expansion for finite population $L$-statistics}\\[2.25\baselineskip]

\small
\text{Andrius {\v C}iginas}\\

\medskip

{\footnotesize
Faculty of Mathematics and Informatics, Vilnius University, LT-03225 Vilnius, Lithuania
}\\[2.25\baselineskip]
\end{center}

\small
\begin{abstract}

   In this paper, we consider the one-term Edgeworth expansion for
   finite population $L$-statistics. We provide an explicit formula for the Edgeworth correction term
   and give sufficient conditions for the validity of the expansion which are expressed in terms of the weight function
   that defines the statistics and moment conditions.
   
\end{abstract}
\vskip 2mm

\normalsize


\vskip 2mm

\noindent\textbf{MSC classes:} 62E20

\section{Introduction}\label{s:1}

Consider a population ${\cal X}=\{x_1,\dots, x_N\}$ of size $N$. We assume, without loss of generality, that $x_1 \leq \dots \leq x_N$. Let $\mathbb X=\{X_1, \dots, X_n\}$ be a simple random sample of size $n<N$ {\it drawn without replacement} from $\cal X$. Let $X_{1:n}\leq \dots \leq X_{n:n}$ be the order statistics of $\mathbb X$. Let $c_1, \dots, c_n$ be a sequence of real numbers. Define the $L$-statistic 
$$
L_n=L_n(\mathbb X)=\frac{1}{n} \sum_{j=1}^n c_j X_{j:n}.
$$
Note that $L_n$ is a symmetric statistic, i.e., for every permutation of $\mathbb X$ the value of $L_n$ remains the same.

In the case where random variables $X_1, \dots ,X_n$ are independent and identically distributed (i.i.d.),
$L$-statistics were studied by a number of authors. 
Asymptotic normality under various conditions was shown in
Chernoff et \al{} \cite{ChGJ_1967}, Shorack \cite{Sh_1972} and Stigler \cite{S_1974}, among others.
Berry-Esseen bounds were obtained in Bjerve \cite{B_1977}, Helmers \cite{H_1977}, and others. 
Later such bounds were generalized to symmetric statistics in van Zwet \cite{Z_1984}.
Large deviations for $L$-statistics were considered in Bentkus and Zitikis \cite{BZ_1990}.
Edgeworth expansions for $L$-statistics were established in Helmers \cite{H_1980}, Bentkus et \al{} \cite{BGZ_1997}, see also Putter \cite{P_1994}.
A general second order asymptotic theory of symmetric asymptotically linear statistics was developed in \cite{BGZ_1997}. Empirical Edgeworth expansions for symmetric statistics were considered in Putter and van Zwet \cite{PZ_1998}.

In the case of samples {\it drawn without replacement} from finite population we know considerably less about $L$-statistics.
As to issues on the asymptotic normality of $L$-statistics we refer to Shao \cite{Sh_1994}.
The one-term Edgeworth expansion for general asymptotically linear symmetric statistics, based on samples drawn
without replacement, is obtained in Bloznelis and G{\" o}tze \cite{BG_2001}. Empirical Edgeworth expansions 
for symmetric finite population statistics are constructed in Bloznelis \cite{B_2001}. 
 
Our analysis is based on the Hoeffding decomposition
\begin{equation}\label{P1}
L_n=\Ex L_n+U_1+\dots +U_n  
\end{equation}
where
\begin{equation*}\label{P2}
U_m=\sum_{1 \leq i_1< \dots <i_m \leq n} g_m(X_{i_1}, \dots ,X_{i_m}), \qquad m=1,\dots ,n.  
\end{equation*}
Here symmetric and centered kernels $g_m, m=1,\dots ,n$ are certain linear combinations of conditional expectations 
\begin{equation*}\label{P3}
h_j(x_{k_1},\dots , x_{k_j})=\Ex \left(L_n-\Ex L_n \,\middle|\, X_1=x_{k_1},\dots ,X_j=x_{k_j}\right),
\quad 1 \leq j \leq m,
\end{equation*}
such that $U_m$ are mutually uncorrelated.
The decomposition in (\ref{P1}) is also called an orthogonal decomposition of $L_n$.
For $U$-statistics, based on \iid{} observations, the orthogonal decomposition was introduced in Hoeffding \cite{H_1948}. In the case of samples drawn without replacement, the orthogonal decompositions of $U$-statistics of the fixed degree $m$ were studied in Zhao and Chen \cite{ZhCh_1990}. In the general case of symmetric statistics (including $L$-statistics), based on samples drawn without replacement, the orthogonal decomposition was studied in Bloznelis and G{\" o}tze \cite{BG_2001}. In particular, \cite{BG_2001} provides expressions for the decomposition kernels as follows 
\begin{align}
\begin{split}
&g_1(x)=\frac{N-1}{N-n}h_1(x),
\end{split}\label{P4} \\
\begin{split}
&g_2(x,y)=\frac{N-2}{N-n}\frac{N-3}{N-n-1}\bigg(h_2(x,y)-\frac{N-1}{N-2}\big(h_1(x)+h_1(y)\big)\bigg),
\end{split}\label{P5} \\
\begin{split}
&g_3(x,y,z)=\frac{N-3}{N-n}\frac{N-4}{N-n-1}\frac{N-5}{N-n-2}\bigg(h_3(x,y,z)-
\frac{N-2}{N-4}\big(h_2(x,y)+h_2(x,z)+h_2(y,z)\big) \\ 
& \hspace{2.75in} +\frac{N-1}{N-3}\frac{N-2}{N-4}\big(h_1(x)+h_1(y)+h_1(z)\big)\bigg).
\end{split}\label{P5_b}
\end{align} 
Clearly $g_1(\cdot), g_2(\cdot, \cdot),\ldots$ depend on the weight sequence $c_1,\ldots ,c_n$ and the distribution
of $X_1$. The first issue, addressed in this paper, is to find a convenient expression of kernels $g_i$, $i=1,2,3$
in terms of the weights $c_1,\dots ,c_n$. For \iid{} samples such a problem is solved in Putter \cite{P_1994}, see also Putter and van Zwet \cite{PZ_1998}. Our purpose is to obtain similar formulas for samples drawn without replacement. Our result is given in Theorem \ref{thm:1} below. 

The second issue, addressed in this paper, is to find specific conditions for $c_1,\ldots ,c_n$ and the distribution
of $X_1$ that ensure the validity of the one-term Edgeworth expansion. This problem is important since it is difficult to check the general conditions given in Theorem 2 of Bloznelis and G{\" o}tze \cite{BG_2001}. 

Define differences of $c_j$ recursively by 
$$
\Delta^0(c_j)=c_j, \quad \Delta^1(c_j)=c_j-c_{j-1}, \quad \Delta^v(c_j)=\Delta^1(\Delta^{v-1}(c_j)),
\quad v=2,3,\ldots, n-1.
$$
Introduce the smoothness conditions for the weights of an $L$-statistic as follows: assume that for some constants $a, b, c$ and $d$ we have
\begin{equation}\label{icase}
\max_{1\leq j \leq n} \left| \Delta^0(c_j) \right| \leq a, \quad
n\max_{2\leq j \leq n} \left| \Delta^1(c_j) \right| \leq b, \quad
n^2\max_{3\leq j \leq n} \left| \Delta^2(c_j) \right| \leq c,
\end{equation}
\begin{equation}\label{iicase}
n^3\max_{4\leq j \leq n} \left| \Delta^3(c_j) \right| \leq d.
\end{equation}
How to check these conditions for arbitrary sample size $n$? The weights of many important $L$-statistics, such as sample mean, trimmed means and Gini's mean difference, may be written in the following way: there exists a function $J\colon(0,1) \to \mathbb R$ such that $c_j=J(j/(n+1))$, $j=1, \dots, n$. Then condition (\ref{icase}) is fulfilled if $J(\cdot)$ has a bounded second derivative. Similarly, boundedness of the third derivative of $J(\cdot)$ implies (\ref{iicase}). Thus, we restrict ourselves to the case of smooth weights. 
For the example of sample mean we have $J\equiv 1$, and, in the case of Gini's mean difference, the weights are of the form $c_j=(n+1)J(j/(n+1))/(n-1)$, $1\leq j\leq n$, with $J(u)=2(2u-1)$. Clearly, in these cases conditions (\ref{icase}) and (\ref{iicase}) are satisfied. A trimmed mean has the weight function $J(u)=(t_2-t_1)^{-1} \, \mathbb I\{ t_1<u<t_2\}$, where $0< t_1<t_2< 1$ are trimming parameters. Here $\mathbb I\{ \cdot\}$ is the indicator function. The smoothness conditions are not fulfilled for this example. For more examples of $L$-statistics with weights generated by the function $J(\cdot)$ we refer to Chernoff et \al{} \cite{ChGJ_1967}.

Consider the normalized $L$-statistic
\begin{equation}\label{Norm}
S_n=n^{\frac{1}{2}}(L_n-\Ex L_n).
\end{equation}
Write 
$$
\tilde{\sigma}_{n}^2=\Var S_n \quad \text{and} \quad \sigma_1^2=\Ex g_1^2(X_1).
$$
Denote
$$
\tau^2=Npq, \quad p=n/N, \quad q=1-p \quad \text{and} \quad n_{*}=\min\{n,N-n\}.
$$
We are going to apply the general result on Edgeworth expansions for symmetric finite population statistics from Bloznelis and G{\" o}tze \cite{BG_2001}.
It follows from Theorem 2 of \cite{BG_2001} that
\begin{equation*}\label{edge_expa}
G_n(x)=\Phi(x)-\frac{(q-p)\alpha+3\kappa}{6\tau}\Phi^{(3)}(x)
\end{equation*}
provides the one-term Edgeworth expansion for the distribution function $F_n(x)={\bf P}\{S_n\leq x\tilde{\sigma}_{n}\}$.
Here $\Phi^{(3)}(x)$ stands for the third derivative of the standard normal distribution function $\Phi(x)$, and
\begin{equation*}\label{alfa_kapa}
\alpha=\sigma_1^{-3} \Ex g_1^3(X_1), \quad \kappa=\sigma_1^{-3}\tau^2 \Ex g_2(X_1,X_2)g_1(X_1)g_1(X_2).
\end{equation*}
Note that, using explicit formulas for $g_1(\cdot)$ and $g_2(\cdot, \cdot)$ given in Theorem \ref{thm:1} below, one can evaluate 
$\alpha$, $\kappa$ and $\sigma_1$.
Define the error of the approximation by 
$$
\Delta_n :=\sup_x\left| F_n(x)-G_n(x) \right|.
$$
Given $g\colon \mathbb{R} \to \mathbb{C}$, write 
$\left\| g \right\|_{[a,b]}=\sup_{a<\left| t \right|<b} \left| g(t) \right|$.
Consider the following nonlattice condition:  
for every $\varepsilon > 0$ and every $B>0$, the function 
$\varphi(t)=\Ex \exp\{it\sigma_1^{-1}g_1(X_1)\}$ satisfies
\begin{equation}\label{sm1}
\liminf_{n_*,N \to \infty} \left\| \varphi \right\|_{[\varepsilon, B]} < 1.
\end{equation}
Consider also the Cramer-type condition
\begin{equation}\label{sm2}
\liminf_{n_*,N \to \infty} \left\| \varphi \right\|_{[\varepsilon, \tau]} < 1.
\end{equation}
Asymptotic conditions (\ref{sm1}) and (\ref{sm2}) are finite population analogues of the nonlattice and Cramer conditions familiar from the traditional case of \iid{} observations. They are imposed on the linear part of the statistic only, as well as in the case of the usual sample mean. In fact, in the \iid{} case, the practical validity of the corresponding smoothness conditions for (infinite) population is more like intuitive, and in finite population settings one can treat them similarly.

Our main result is the following statement.
\begin{thm}\label{thm:2}
Assume that $\tilde{\sigma}_{n}$ remains bounded away from zero as $n \to \infty$.  

(i) Assume that \eqref{icase} and \eqref{sm1} hold and that $\Ex \left| X_1 \right|^{3+\delta} < \infty$
for some $\delta > 0$. Then $\Delta_n=o(n_*^{-1/2})$ as $n_*,N \to \infty$. 

(ii) Assume that \eqref{iicase} and \eqref{sm2} hold and
$\Ex \left| X_1 \right|^4 < \infty$. Then $\Delta_n=O(n_*^{-1})$ as $n_*,N \to \infty$.
\end{thm}

Note that, conditions of Theorem \ref{thm:2} are similar to that used in the \iid{} situation,
see Helmers \cite{H_1980}, Putter \cite{P_1994} and Bentkus et \al{} \cite{BGZ_1997}.
Our proofs are purely combinatorial and differ from those used in the case of the \iid{} observations. 
In Section \ref{s:2}, we get explicit expressions for the first three terms of orthogonal decomposition. 
In Section \ref{s:3}, we prove Theorem \ref{thm:2}. The performance of the Edgeworth expansion is illustrated by a numerical example in Section \ref{s:5}. 
Proofs of lemmas are given in Section \ref{s:4}.

\section{Orthogonal decomposition}\label{s:2}

Given $0 \leq m \leq n$ and $1 \leq k_1< \dots <k_m \leq N$, introduce the event
$$
A_m=A_{x_{k_1}\cdots x_{k_m}}=\{X_1=x_{k_1},\dots ,X_m=x_{k_m}\}.
$$
For convenience of notation, we define $k_0:=0$ and $k_{m+1}:=N+1$.
Introduce numbers $x_0:=x_1$ and $x_{N+1}:=x_N$ 
and define $X_{0:n}:=x_0$ and $X_{n+1:n}:=x_{N+1}$, so that, almost surely, we have 
$X_{0:n}\leq X_{j:n}\leq X_{n+1:n}$ for each $1\leq j\leq n$.
In the proof of Theorem \ref{thm:1} below we represent order statistics by sums of sample spacings
\begin{equation}\label{spac_dec}
X_{j:n}=\sum_{r=0}^{j-1} \mathbf{\Delta}_{r:n} + x_0, \quad j=1,\dots n.
\end{equation}
Here $\mathbf{\Delta}_{r:n}=X_{r+1:n}-X_{r:n}$, $r=0,\dots ,n$ denote the sample spacings.
Write $\vartriangle_i=x_{i+1}-x_{i}$, $i=0,\dots ,N$. 

\begin{lem}\label{lem:1}
For any $m=0,\dots ,n$ and $r=0,\dots ,n$ we have
\begin{equation}\label{P6}
\Ex \left(\mathbf{\Delta}_{r:n} \,\middle|\, A_m\right)={N-m\choose n-m}^{-1}\,\sum_{s=1}^{m+1} \sum_{i=k_{s-1}}^{k_{s}-1} {i-s+1\choose r-s+1}{N-i-m+s-1\choose n-r-m+s-1}\vartriangle_i.
\end{equation}
\end{lem}

Denote by $\mathcal{H}_{N,n,i}(j)={i\choose j}{N-i\choose n-j}/{N\choose n}$ the probability that 
a hypergeometric random variable with parameters $N$, $n$ and $i$ attains the value $j$.  
Denote $[N]_j=N(N-1)\cdots (N-j+1)$.
Next we give explicit and comparatively simple expressions of kernels (\ref{P4}), (\ref{P5}) and (\ref{P5_b}).

\begin{thm}\label{thm:1} 
\mbox{}

(i) For $1\leq k \leq N$ 
\begin{equation}\label{P14}
g_1(x_k)=-n^{-1}\sum_{j=1}^{n} \Delta^0(c_j) \sum_{i=1}^{N-1} \varphi_{k}(i) 
\mathcal{H}_{N-2,n-1,i-1}(j-1)\vartriangle_i,
\end{equation} 
where 
\begin{equation*}
\varphi_{k}(i)=\begin{cases}
 -i/N &\text{if $1\leq i<k$,}\\
1-i/N &\text{if $k\leq i<N$.}
\end{cases}
\end{equation*}

(ii) For $1\leq k<l \leq N$
\begin{equation}\label{P15}
g_2(x_k,x_l)=-n^{-1}\sum_{j=2}^{n} \Delta^1(c_j) \sum_{i=1}^{N-1} \phi_{k,l}(i) 
\mathcal{H}_{N-4,n-2,i-2}(j-2)\vartriangle_i,
\end{equation} 
where 
\begin{equation}\label{cof2}
\phi_{k,l}(i)=\begin{cases}
i(i-1)/[N-1]_2 &\text{if $1\leq i<k$,}\\
-(i-1)(N-i-1)/[N-1]_2 &\text{if $k\leq i<l$,}\\
(N-i-1)(N-i)/[N-1]_2 &\text{if $l\leq i<N$.}
\end{cases}
\end{equation}

(iii) For $1\leq k<l<m \leq N$
\begin{equation}\label{P15_b}
g_3(x_k,x_l,x_m)=-n^{-1}\sum_{j=3}^{n} \Delta^2(c_j) \sum_{i=1}^{N-1} \theta_{k,l,m}(i) 
\mathcal{H}_{N-6,n-3,i-3}(j-3)\vartriangle_i,
\end{equation} 
where 
\begin{equation}\label{cof3}
\theta_{k,l,m}(i)=\begin{cases}
-i(i-1)(i-2)/[N-2]_3 &\text{if $1\leq i<k$,}\\
(i-1)(i-2)(N-i-2)/[N-2]_3 &\text{if $k\leq i<l$,}\\
-(i-2)(N-i-2)(N-i-1)/[N-2]_3 &\text{if $l\leq i<m$,}\\
(N-i-2)(N-i-1)(N-i)/[N-2]_3 &\text{if $m\leq i<N$.}
\end{cases}
\end{equation}

\end{thm}

\begin{proof} 
\emph{(i)}
First we write a kernel of orthogonal decomposition of the order statistic $X_{j:n}$, $j=1,\ldots ,n$.
For chosen $1\leq k \leq N$, using representation (\ref{spac_dec})
and Lemma \ref{lem:1} for $m=0,1$, we have
\begin{equation*}\label{P16}
g_{1j}(x_k)={N-2\choose n-1}^{-1} \,\sum_{r=0}^{j-1} 
\left\{ \sum_{i=1}^{k-1} \frac{i}{N} \theta_{21}(i,r) \vartriangle_i -
\sum_{i=k}^{N-1} \left(1-\frac{i}{N}\right) \theta_{22}(i,r) \vartriangle_i \right\},   
\end{equation*}
where
\begin{align*}
&\theta_{21}(i,r)=\frac{N}{i}{i\choose r}\left\{ {N-i-1\choose n-r-1}-
\frac{n}{N}{N-i\choose n-r} \right\} \\
\intertext{and}
&\theta_{22}(i,r)=-\frac{N}{N-i}{N-i\choose n-r}\left\{ {i-1\choose r-1}-
\frac{n}{N}{i\choose r} \right\}.
\end{align*}
It is easy to verify that $\theta_{21}(i,r)=\theta_{22}(i,r)$.
Next, using induction it is easy to show that for
every $j=1,\dots ,n$ 
$$
\sum_{r=0}^{j-1} \theta_{22}(i,r)={i-1\choose j-1}{N-i-1\choose n-j}
$$
and the proof of part \emph{(i)} follows from a simple observation that
$$
g_1(x_k)=n^{-1}\sum_{j=1}^n c_j g_{1j}(x_k), \quad 1\leq k\leq N.
$$

\emph{(ii)} 
Similarly, for chosen $1\leq k<l \leq N$, using representation (\ref{spac_dec})
and Lemma \ref{lem:1} for $m=0,1,2$, for a single order statistic we get
\begin{equation*}\label{P17}
\begin{split}
g_{2j}(x_k,x_l)={N-4\choose n-2}^{-1} \,\sum_{r=0}^{j-1} 
&\left\{ \sum_{i=1}^{k-1} \frac{i(i-1)}{[N-1]_2} \theta_{31}(i,r) \vartriangle_i -
\sum_{i=k}^{l-1} \frac{(i-1)(N-i-1)}{[N-1]_2} \theta_{32}(i,r) \vartriangle_i \right. {}\\&+ 
\left. \sum_{i=l}^{N-1} \frac{(N-i)(N-i-1)}{[N-1]_2} \theta_{33}(i,r) \vartriangle_i \right\},   
\end{split}
\end{equation*}
where
\begin{align*}
\begin{split}
\theta_{31}(i,r)=&\frac{[N-1]_2}{i(i-1)}{i\choose r}
\left\{
{N-i-2\choose n-r-2}-
2\frac{n-1}{N-2}{N-i-1\choose n-r-1} +
\frac{n(n-1)}{[N-1]_2}{N-i\choose n-r}
\right\}
\end{split} \\
\intertext{and}
\begin{split}
\theta_{32}(i,r)=&-\frac{[N-1]_2}{(i-1)(N-i-1)}
\left[ 
{i-1\choose r-1}
\left\{ {N-i-1\choose n-r-1}-\frac{n-1}{N-2}{N-i\choose n-r} \right\} \right. {}\\& \hspace{1.5in}
-\left. \frac{n-1}{N-2}{i\choose r}
\left\{ {N-i-1\choose n-r-1}-\frac{n}{N-1}{N-i\choose n-r} \right\}
\right],
\end{split} \\
\intertext{and}
\begin{split}
\theta_{33}(i,r)=&\frac{[N-1]_2}{(N-i)(N-i-1)}{N-i\choose n-r}
\left\{
{i-2\choose r-2}-
2\frac{n-1}{N-2}{i-1\choose r-1} +
\frac{n(n-1)}{[N-1]_2}{i\choose r}
\right\}.
\end{split}
\end{align*}
Similarly $\theta_{31}(i,r)=\theta_{32}(i,r)=\theta_{33}(i,r)$.
Next, using induction one can show that, for every $j=1,\dots ,n$, 
$$
\sum_{r=0}^{j-1} \theta_{33}(i,r)=
{i-2\choose j-1}{N-i-2\choose n-j-1}-{i-2\choose j-2}{N-i-2\choose n-j}.
$$
To complete the proof of part \emph{(ii)} we observe that
$$
g_2(x_k,x_l)=n^{-1}\sum_{j=1}^n c_j g_{2j}(x_k,x_l), \quad 1\leq k<l\leq N
$$
and apply 
$$\sum_{j=1}^n c_j(b_{j+1}-b_{j})=c_n b_{n+1}-c_1 b_1-\sum_{j=2}^n (c_j-c_{j-1})b_j, \quad \text{where} \quad 
b_j={i-2\choose j-2}{N-i-2\choose n-j}.
$$

\emph{(iii)} The proof of this part is very similar.
\end{proof}

\section{Edgeworth expansion}\label{s:3}

Decomposition (\ref{P1}) is a stochastic expansion of an $L$-statistic and the first few terms of the decomposition can be quite an excellent approximation to $L_n$. In order to control the accuracy of approximation, one can use the smoothness conditions defined below.

Let $(X_1,\dots,X_N)$ denote a random permutation of the ordered set $(x_1,\dots,x_N)$ which is uniformly distributed over the class of permutations. Then, the first $n$ observations $X_1, \dots, X_n$ represent a simple random sample from population $\cal X$. For $j=1,\dots,N-n$ denote $X_{j}^{'}=X_{n+j}$.
Define 
$$
D^{j}L_n=L_n(X_1,\dots,X_n)-L_n(X_1,\dots,X_{j-1},X_{j+1},\dots,X_n,X_{j}^{'}).
$$
Higher order difference operations are defined recursively:
$$
D^{j_1,j_2}L_n=D^{j_2}(D^{j_1}L_n), \qquad D^{j_1,j_2,j_3}L_n=D^{j_3}(D^{j_2}(D^{j_1}L_n)),\dots \quad .
$$ 
They are symmetric; that is, $D^{j_1,j_2}L_n=D^{j_2,j_1}L_n$, etc. Write
$$
\delta_k=\delta_k(L_n)=\Ex \left(n_*^{(k-1)}{\mathbb D}_kL_n\right)^2, \qquad {\mathbb D}_kL_n=D^{1,2,\dots,k}L_n, \quad
1 \leq k < n_*.
$$

We shall write the differences ${\mathbb D}_kL_n$, $k=1,2,3,4$ in terms of differences of the weights $c_1,\ldots ,c_n$. 
For that purpose we need additional notation.
For $k=1,2,3,4$, denote by $X_{1:n+k}\leq \dots \leq X_{n+k:n+k}$ order statistics which correspond to the sample ${\mathbb X}_k=\{X_1,\dots ,X_{n+k}\}$. Let ${\mathbb R}_k=\{R_1,\dots ,R_{n+k}\}$ be the ranks of the sample ${\mathbb X}_k$ assigned as follows: introduce the function $f\colon{\cal X} \to \{1,\ldots ,N\}$ given by $f(x_k)=k$, $k=1, \dots, N$.
Then we decide that $R_i<R_j$ if $f(X_i)<f(X_j)$, and so on. Thus ranks ${\mathbb R}_k$ are all distinct, i.e., ${\mathbb R}_k$ is a random permutation of the set $\{1,\dots ,n+k\}$. 
Further, denote ${\mathbb R}_k^{*}=\{R_1,\dots ,R_k,R_{n+1},\dots ,R_{n+k}\}$
and let ${\mathcal R}_k$ be a set of all permutations of the set ${\mathbb R}_k^{*}$, where a particular permutation means the arrangement of elements of the set by size. Let $R_{1:2k}< \dots <R_{k:2k}<R_{n+1:2k}< \dots <R_{n+k:2k}$ denote order statistics which correspond to ${\mathbb R}_k^{*}$.  

\begin{lem}\label{lem:3}

For each $k=1,2,3,4$ there exists a random variable $d_k=d_k({\mathbb R}_k^{*})$ with values in $\{-1,0,1\}$ such that
\begin{equation*}
{\mathbb D}_k L_n=d_k n^{-1}\sum_{j=R_{k:2k}}^{R_{n+1:2k}-1} \Delta^{k-1}(c_j)\mathbf{\Delta}_{j:n+k}.
\end{equation*}  

\end{lem}
     
Consider further the normalized $L$-statistic $S_n$, defined in (\ref{Norm}).
It will be seen from the proof of Theorem \ref{thm:2} below how conditions (\ref{icase}) and (\ref{iicase}) for the weight sequence $c_1,\ldots ,c_n$ (together with an appropriate moment condition for $X_1$) correspond to the smoothness conditions
\begin{equation*}\label{SmC}
\delta_{k}(S_n)=O(n_*^{-1}) \quad \text{as} \quad n_*,N \to \infty,  
\end{equation*}
for $k=3,4$. 

For the statistic $S_n$ from Theorem \ref{thm:1} we have kernels $g_1$, $g_2$ and $g_3$ which differ from (\ref{P14}), (\ref{P15}) and (\ref{P15_b}) only by the multiplier $n^{-1/2}$. 
Write
$$
\beta_s=\Ex \left| n_*^{1/2}g_1(X_1)\right|^s, \quad 
\gamma_s=\Ex \left| n_*^{3/2}g_2(X_1,X_2)\right|^s, \quad
\zeta_s=\Ex \left| n_*^{5/2}g_3(X_1,X_2,X_3)\right|^s.
$$
Theorem \ref{thm:2} is the consequence of the following theorem.
\begin{thmBG}[Theorem 2 of Bloznelis and G{\" o}tze \cite{BG_2001}]
Assume that $\tilde{\sigma}_{n}$ remains bounded away from zero as $n \to \infty$.  

(i) Assume that \eqref{sm1} holds, $\delta_3=o(n_*^{-1/2})$ and for some $\delta > 0$, the moments $\beta_{3+\delta}$ and $\gamma_{2+\delta}$ are bounded as $n_*,N \to \infty$. Then
$\Delta_n=o(n_*^{-1/2})$ as $n_*,N \to \infty$.

(ii) Assume that \eqref{sm2} holds, $\delta_4=O(n_*^{-1})$ and, the moments $\beta_4$, $\gamma_4$, $\zeta_2$ 
are bounded as $n_*,N \to \infty$. Then
$\Delta_n=O(n_*^{-1})$ as $n_*,N \to \infty$.
\end{thmBG}

Let us turn to the proof of our main result.

\begin{proof}[Proof of Theorem \ref{thm:2}]
Following van Zwet \cite{Z_1984} we introduce the functions
\begin{equation}\label{specf}
\begin{split}
&G(x)=\int_{-\infty}^{x} F(y)\, dy, \quad
H(x)=\int_{x}^{+\infty} (1-F(y))\, dy, \\
&M(x)=\int_{-\infty}^{x} F(y)(1-F(y))\, dy,
\end{split}
\end{equation}
where 
\begin{equation}\label{df}
F(y)=\frac{1}{N}\sum_{i=1}^N \mathbb{I}{\{x_i\leq y\}}
\end{equation}
is the distribution function of the random variable $X_1$. 

A simple integration of (\ref{df}) and some work with sums yield that at the points $x=x_k$, $k=1,\dots ,N$
we have
\begin{equation*}\label{specf_k}
\begin{split}
&G(x_k)=\sum_{i=1}^{k-1} \frac{i}{N}\vartriangle_i, \quad
H(x_k)=\sum_{i=k}^{N-1} \left(1-\frac{i}{N}\right)\vartriangle_i, \\
&M(x_k)=\sum_{i=1}^{k-1} \frac{i}{N}\left(1-\frac{i}{N}\right)\vartriangle_i.
\end{split}
\end{equation*}
Functions (\ref{specf}) are finite and monotone, and, similarly as in Putter \cite{P_1994}, we obtain for $k=1,\dots ,N$
\begin{equation}\label{ineq1}
G(x_k)+H(x_k)\leq \Ex \left| X_1 \right| + \left| x_k \right|
\end{equation}
and
\begin{equation}\label{ineq2}
M(x_N)\leq G(x_k)+H(x_k) \leq \Ex \left| X_1 \right| + \left| x_k \right|.
\end{equation}
One can show (we omit a detailed proof) that the quantities $\phi_{k,l}(i)$ and $\theta_{k,l,m}(i)$, defined by
(\ref{cof2}) and (\ref{cof3}), satisfy
\begin{equation}\label{ineq_cof2}
\lvert \phi_{k,l}(i) \rvert \leq \left( \frac{N}{N-1} \right)^2 \lvert \phi_{k,l}^{'}(i) \rvert \leq
4\lvert \phi_{k,l}^{'}(i) \rvert 
\end{equation}
and
\begin{equation}\label{ineq_cof3}
\lvert \theta_{k,l,m}(i) \rvert \leq \left( \frac{N}{N-2} \right)^3 \lvert \theta_{k,l,m}^{'}(i) \rvert \leq
27\lvert \theta_{k,l,m}^{'}(i) \rvert,
\end{equation}
where we write
$$
\phi_{k,l}^{'}(i)=\left(\mathbb{I}{\{i\geq k\}}-i/N\right)
\left(\mathbb{I}{\{i\geq l\}}-i/N\right)
$$
and
$$
\theta_{k,l,m}^{'}(i)=\left(\mathbb{I}{\{i\geq k\}}-i/N\right)
\left(\mathbb{I}{\{i\geq l\}}-i/N\right)\left(\mathbb{I}{\{i\geq m\}}-i/N\right).
$$

Then, by (\ref{icase}), (\ref{ineq1}) and because 
$\sum_{j=1}^n \mathcal{H}_{N-2,n-1,i-1}(j-1)=1$, for $1\leq k \leq N$,  
\begin{equation*}
\begin{split}
\left| g_1(x_k) \right| &\leq  
a n^{-\frac{1}{2}}\sum_{i=1}^{N-1} \lvert \varphi_{k}(i) \rvert \vartriangle_i=
a n^{-\frac{1}{2}}\left[G(x_k)+H(x_k)\right] \leq 
a n^{-\frac{1}{2}}\left[\Ex \left| X_1 \right| + \left| x_k \right| \right],
\end{split}
\end{equation*}
and thus for $s \geq 1$,
\begin{equation}\label{bet_est}
\begin{split}
\Ex \left| g_1(X_1) \right| ^s &=
\frac{1}{N} \sum_{k=1}^N \left| g_1(x_k) \right| ^s \leq  
a^s n^{-\frac{s}{2}} \frac{1}{N} \sum_{k=1}^N \left[\Ex \left| X_1 \right| + \left| x_k \right| \right]^s {}\\
&\leq a^s n^{-\frac{s}{2}} \frac{1}{N} \sum_{k=1}^N 2^{s-1} \left[\left(\Ex \left| X_1 \right| \right)^s + 
\left| x_k \right| ^s \right]
\leq 2^s a^s n^{-\frac{s}{2}} \Ex \left| X_1 \right| ^s.
\end{split}
\end{equation}

Also, by (\ref{icase}), (\ref{ineq_cof2}), (\ref{ineq1}), (\ref{ineq2}),
because of monotonicity of $G$, $H$, $M$ and 
$\sum_{j=2}^n \mathcal{H}_{N-4,n-2,i-2}(j-2)=1$, for $1\leq k<l \leq N$, 
\begin{equation*}
\begin{split}
\left| g_2(x_k,x_l) \right| &\leq  
b n^{-\frac{3}{2}}\sum_{i=1}^{N-1} \lvert \phi_{k,l}(i) \rvert \vartriangle_i \leq 
4 b n^{-\frac{3}{2}}\sum_{i=1}^{N-1} \lvert \phi_{k,l}^{'}(i) \rvert \vartriangle_i \\ 
&\leq 4 b n^{-\frac{3}{2}}\left[G(x_l)+M(x_N)+H(x_l) \right] \leq 
8 b n^{-\frac{3}{2}}\left[\Ex \left| X_1 \right| + \left| x_l \right| \right],
\end{split}
\end{equation*}
and then for $s \geq 1$,
\begin{equation}\label{gam_est}
\begin{split}
\Ex \left| g_2(X_1,X_2) \right| ^s &=
{N\choose 2}^{-1} \sum_{1\leq k<l\leq N} \left| g_2(x_k,x_l) \right| ^s \leq   
2^{3s} b^s n^{-\frac{3}{2} s} {N\choose 2}^{-1} \sum_{1\leq k<l\leq N} 
\left[\Ex \left| X_1 \right| + \left| x_l \right| \right]^s {}\\
&\leq 2^{3s} b^s n^{-\frac{3}{2} s} {N\choose 2}^{-1} \sum_{1\leq k<l\leq N}
2^{s-1} \left[\left(\Ex \left| X_1 \right| \right)^s + \left| x_l \right| ^s \right] \leq 
2^{4s+1} b^s n^{-\frac{3}{2} s} \Ex \left| X_1 \right| ^s.
\end{split}
\end{equation}

Similarly, by (\ref{icase}), (\ref{ineq_cof3}), (\ref{ineq1}), (\ref{ineq2}),
because of monotonicity of $G$, $H$, $M$ and 
$\sum_{j=3}^n \mathcal{H}_{N-6,n-3,i-3}(j-3)=1$, for $1\leq k<l<m \leq N$, 
\begin{equation*}
\begin{split}
\left| g_3(x_k,x_l,x_m) \right| &\leq  
3^4 c n^{-\frac{5}{2}}\left[\Ex \left| X_1 \right| + \left| x_m \right| \right],
\end{split}
\end{equation*}
and then for $s \geq 1$,
\begin{equation}\label{zet_est}
\begin{split}
\Ex \left| g_3(X_1,X_2,X_3) \right| ^s & \leq  
2^{s+2} 3^{4s} c^s n^{-\frac{5}{2} s} \Ex \left| X_1 \right| ^s.
\end{split}
\end{equation}

Note that, it follows from inequalities (\ref{bet_est}), (\ref{gam_est}), (\ref{zet_est}) that, for both cases \emph{(i)} and \emph{(ii)}, 
the moments $\beta_s$, $\gamma_s$ and $\zeta_s$ in Theorem BG are bounded if the corresponding moments $\Ex \left| X_1 \right| ^s$ are finite.

Next we evaluate the quantities $\delta_3(S_n)$ and $\delta_4(S_n)$.
Let us introduce the event $\mathfrak{R}_{k;ij}=\{ R_{k:2k}=i, R_{n+1:2k}=j\}$.
Application of Lemma \ref{lem:3} and the use of the corresponding smoothness conditions (\ref{icase}) and (\ref{iicase})  for $k=3,4$ yield
\begin{equation}\label{exp_con}
\begin{split}
\Ex \left[(n_*^{(k-1)}{\mathbb D}_kS_n)^2 \,\middle|\, \mathfrak{R}_{k;ij} \right] \leq 
C_1 n^{-1} \Ex \left[(X_{j:n+k}-X_{i:n+k})^2 \,\middle|\, \mathfrak{R}_{k;ij} \right],
\end{split}
\end{equation}
where we denote $C_1=\max\left\{ c^2,d^2 \right\}$.
It is a simple problem to calculate that
\begin{equation*}
p_{k;ij}=\mathbf{P}\left\{ \mathfrak{R}_{k;ij} \right\}=
{i-1\choose k-1}{n+k-j\choose k-1}\left/{n+k\choose 2k}\right. ,
\end{equation*}
where $k\leq i<j \leq n+1$. 
Since the events $\mathfrak{R}_{k;ij}$ and $\mathfrak{B}_{k;ijlm}=\left\{ X_{i:n+k}=x_l, X_{j:n+k}=x_m  \right\}$ are independent, for $x_1<\cdots <x_N$ we have
\begin{equation*}
\begin{split}
p_{k;ijlm}&=\mathbf{P}\left\{ \mathfrak{B}_{k;ijlm} \,\middle|\, \mathfrak{R}_{k;ij} \right\}=
{l-1\choose i-1}{m-l-1\choose j-i-1}{N-m\choose n+k-j}\left/{N\choose n+k}\right. ,
\end{split}
\end{equation*}
where $i \leq l<m \leq N-n-k+j$ and $m-l \geq j-i$.
For $x_1\leq \cdots \leq x_N$ these probabilities are the same. It follows from the same argument as in the proof of Lemma \ref{lem:1}.

Then, we derive from (\ref{exp_con})
\begin{equation*}\label{ncc}
\begin{split}
&\delta_k(S_n) \leq C_1 n^{-1} \sum_{1 \leq l<m \leq N} \lambda_{k;lm} (x_m-x_l)^2, \quad \text{where} \quad 
\lambda_{k;lm}=\sum_{k \leq i<j \leq n+1} p_{k;ij}p_{k;ijlm}.
\end{split}
\end{equation*}
It is easy to check that
$$
\Var X_1=\frac{1}{N^2}\sum_{1\leq l<m \leq N} (x_m-x_l)^2.
$$
Clearly, to prove the bounds $\delta_k(S_n)=O(n_*^{-1})$, $k=3,4$, it will now suffice to show that
$\lambda_{k;lm}=O(N^{-2})$ for all $1\leq l<m \leq N$.

Using the simple inequality ${u\choose v} \leq u^v/v!$ for all $k\leq i<j \leq n+1$, we obtain
\begin{equation}\label{pij}
\begin{split}
{i-1\choose k-1}{n+k-j\choose k-1} & \leq
\frac{(i-1)^{k-1}}{(k-1)!} \frac{(n+k-j)^{k-1}}{(k-1)!} \leq
\frac{[(i-1)(n+k-1-i)]^{k-1}}{[(k-1)!]^2} {}\\& =
\frac{(n+k-2)^{2(k-1)}}{[(k-1)!]^2} 
\left [\frac{i-1}{n+k-2} \left (1-\frac{i-1}{n+k-2} \right ) \right ]^{k-1} {}\\& \leq
\left(\frac{1}{4} \right)^{k-1} \frac{(n+k-2)^{2(k-1)}}{[(k-1)!]^2}. 
\end{split}
\end{equation}
Also, for all $1\leq l<m \leq N$
\begin{equation}\label{pijlm}
\begin{split}
\sum_{k\leq i<j \leq n+1} p_{k;ijlm}&=
{N\choose n+k}^{-1} \,\sum_{s=k-1}^{n-1} \sum_{t=0}^{n-1-s}
{l-1\choose s}{m-l-1\choose t}{N-m\choose n+k-2-s-t} {}\\& \leq
{N\choose n+k}^{-1} \,\sum_{s=0}^{n+k-2} \sum_{t=0}^{n+k-2-s}
{l-1\choose s}{m-l-1\choose t}{N-m\choose n+k-2-s-t} {}\\& =
{N\choose n+k}^{-1} {N-2\choose n+k-2}.
\end{split}
\end{equation}
In the last step here the generalized Vandermonde identity was applied. 

Finally, from (\ref{pij}) and (\ref{pijlm}) we derive
\begin{equation*}
\begin{split}
\lambda_{k;lm} & \leq
{n+k\choose 2k}^{-1} \left(\frac{1}{4} \right)^{k-1} \frac{(n+k-2)^{2(k-1)}}{[(k-1)!]^2}
{N\choose n+k}^{-1} {N-2\choose n+k-2} \leq  
C_k N^{-2}, 
\end{split}
\end{equation*}
for all $1\leq l<m \leq N$, where 
\begin{equation*}
C_k=2 \left(\frac{1}{4} \right)^{k-1} \frac{(2k)!}{[(k-1)!]^2}
\frac{(2k-2)^{2k-3}}{(2k-3)!},  \quad k=3,4.
\end{equation*}
Application of Theorem BG completes the proof of both cases \emph{(i)} and \emph{(ii)} of the theorem.
\end{proof}

\section{An example}\label{s:5}

We show, by the following simple example, how the Edgeworth expansion improves the usual normal approximation. A population ${\cal X}$ of size $N=100$ was simulated from the logistic distribution with the distribution function $H(x)=(1+e^{-x})^{-1}$ for $-\infty<x<+\infty$. Our chosen population ${\cal X}$ has mean $0.004$ and variance $3.270$. Consider the $L$-statistic with the weights $c_j=J(j/(n+1))$, $1\leq j\leq n$, where $J(u)=6u(1-u)$. Similarly as in the case of \iid{} observations (see Chernoff et \al{} \cite{ChGJ_1967}), the defined statistic may be applied in estimation of a center of the population.

For samples of sizes $n=5, 15, 30$, we present several $q$-quantiles of the functions $\tilde{F}_n$, $G_n$ and $\Phi$ in Table \ref{t1} below. Here $\tilde{F}_n$ is the Monte--Carlo approximation of $F_n$, obtained by drawing (independently) $10^7$ samples without replacement from ${\cal X}$.
\begin{table}[!h]
\caption{Approximations to $\tilde{F}_n$, $n=5, 15, 30$}\label{t1}
 \centering
 \tabcolsep=7pt
 \vspace{2mm}
\begin{tabular}{rrrrrrrrrrr}
\boldhline
     $q=$                        & 0.01 & 0.05 & 0.10 & 0.25 & 0.50 & 0.75 & 0.90 & 0.95 & 0.99 \\
\hline
$\tilde{F}_{5}^{-1}(q)\approx$   & -2.120 & -1.534 & -1.224 & -0.692 & -0.063 & 0.633 & 1.319 & 1.750 & 2.581 \\

$G_{5}^{-1}(q)\approx$           & -2.078 & -1.557 & -1.249 & -0.704 & -0.058 & 0.639 & 1.323 & 1.754 & 2.563 \\

$\tilde{F}_{15}^{-1}(q)\approx$  & -2.159 & -1.577 & -1.253 & -0.694 & -0.040 & 0.653 & 1.308 & 1.712 & 2.489 \\

$G_{15}^{-1}(q)\approx$          & -2.160 & -1.585 & -1.259 & -0.695 & -0.039 & 0.652 & 1.308 & 1.714 & 2.487 \\

$\tilde{F}_{30}^{-1}(q)\approx$  & -2.199 & -1.604 & -1.268 & -0.689 & -0.026 & 0.663 & 1.299 & 1.688 & 2.428 \\

$G_{30}^{-1}(q)\approx$          & -2.220 & -1.606 & -1.267 & -0.688 & -0.025 & 0.660 & 1.297 & 1.687 & 2.430 \\

$\Phi^{-1}(q)\approx$            & -2.326 & -1.645 & -1.282 & -0.674 & \phantom{-}0.000 & \phantom{-}0.674 & \phantom{-}1.282 & \phantom{-}1.645 & \phantom{-}2.326 \\
\boldhline
\end{tabular}  
\end{table}

Table \ref{t1} shows that, even for small sample of size $n=5$, the Edgeworth expansion is much more efficient than the normal approximation. Note that, in practice, the population parameters $\alpha$ and $\kappa$, which define the Edgeworth expansion $G_n$, should be estimated, see, e.g., Bloznelis \cite{B_2001}.

\section{Proofs of lemmas}\label{s:4}

\begin{proof}[Proof of Lemma \ref{lem:1}]
Assume that $x_1<\cdots <x_N$. Then, for any $m=0,\dots ,n$ and $r=0,\dots ,n+1$, straightforward combinatorial calculations give
\begin{equation*}\label{P7}
\begin{split}
\Ex \left( X_{r:n} \,\middle|\, A_m \right)=&{N-m\choose n-m}^{-1}\Bigg[\sum_{s=1}^{m+1} \sum_{i=k_{s-1}+1}^{k_{s}-1} 
{i-s\choose r-s}{N-i-m+s-1\choose n-r-m+s-1}x_{i} \\
+& \sum_{s=0}^{m+1}{k_{s}-s\choose r-s}{N-k_{s}-m+s\choose n-r-m+s} x_{k_s}\Bigg].
\end{split}
\end{equation*}
The key idea is for $r=0,\dots ,n$ to note that
\begin{equation*}\label{P8}
\begin{split}
\Ex \left( X_{r+1:n} \,\middle|\, A_m \right)=&{N-m\choose n-m}^{-1}\Bigg[\sum_{s=1}^{m+1} \sum_{i=k_{s-1}+1}^{k_{s}-1} 
{i-s\choose r-s+1}\delta_{m,s,i}^{'}(r) x_{i} \\
+& \sum_{s=0}^{m+1}{k_{s}-s\choose r-s+1}{N-k_{s}-m+s\choose n-r-m+s-1} x_{k_s}\Bigg],
\end{split}
\end{equation*}
where
$$
\delta_{m,s,i}^{'}(r)={N-i-m+s\choose n-r-m+s-1}-{N-i-m+s-1\choose n-r-m+s-1}
$$
and
\begin{equation*}\label{P9}
\begin{split}
\Ex \left( X_{r:n} \,\middle|\, A_m \right)=&{N-m\choose n-m}^{-1}\Bigg[\sum_{s=1}^{m+1} \sum_{i=k_{s-1}+1}^{k_{s}-1} 
\delta_{m,s,i}^{''}(r){N-i-m+s-1\choose n-r-m+s-1} x_{i} \\
+& \sum_{s=0}^{m+1}{k_{s}-s\choose r-s}{N-k_{s}-m+s\choose n-r-m+s} x_{k_s}\Bigg],
\end{split}
\end{equation*}
where
$$
\delta_{m,s,i}^{''}(r)={i-s+1\choose r-s+1}-{i-s\choose r-s+1}.
$$
Then, it is easy to verify that, for $r=0,\dots ,n$,
$\Ex \left( \mathbf{\Delta}_{r:n} \,\middle|\, A_m \right)$ is the same as in (\ref{P6}).
For $x_1\leq \cdots \leq x_N$ this result does not change if we assume (without loss of generality) that coincident values of ${\cal X}$ are strictly ordered by their unique names, e.g., by sizes of their indexes.
\end{proof}

\begin{proof}[Proof of Lemma \ref{lem:3}]
The set $\mathcal R_1$ contains only $2$ elements, therefore we elaborate the case $k=1$.

Let $R_1<R_{n+1}$. Then
\begin{equation*}
\begin{split}
n\mathbb D_1 L_n&=n\big(L_n(\mathbb X_1\backslash{\{X_{n+1}\}})-L_n(\mathbb X_1\backslash{\{X_1\}})\big) \\
&=\Bigg[ \sum_{j=1}^{R_{n+1}-1} c_j X_{j:n+1} + \sum_{j=R_{n+1}}^{n} c_j X_{j+1:n+1} \Bigg] 
-\Bigg[ \sum_{j=1}^{R_{1}-1} c_j X_{j:n+1} + \sum_{j=R_{1}}^{n} c_j X_{j+1:n+1} \Bigg] \\
&= -\sum_{j=R_1}^{R_{n+1}-1} c_j\mathbf{\Delta}_{j:n+1}.
\end{split}
\end{equation*}

Let $R_1>R_{n+1}$. Then
\begin{equation*}
\begin{split}
n\mathbb D_1 L_n&=n\big(L_n(\mathbb X_1\backslash{\{X_{n+1}\}})-L_n(\mathbb X_1\backslash{\{X_1\}})\big) \\
&=\Bigg[ \sum_{j=1}^{R_{n+1}-1} c_j X_{j:n+1} + \sum_{j=R_{n+1}}^{n} c_j X_{j+1:n+1} \Bigg] 
-\Bigg[ \sum_{j=1}^{R_{1}-1} c_j X_{j:n+1} + \sum_{j=R_{1}}^{n} c_j X_{j+1:n+1} \Bigg] \\
&= \sum_{j=R_{n+1}}^{R_{1}-1} c_j\mathbf{\Delta}_{j:n+1}.
\end{split}
\end{equation*}

Next we calculate $\mathbb D_k L_n$, $k=2,3,4$ recursively. Write
$\mathbb D_2 L_n=\mathbb D_1^{'} L_n-\mathbb D_1^{''} L_n$, where
\begin{align*}
\mathbb D_1^{'} L_n&=L_n(\mathbb X_2\backslash{\{X_{n+1},X_{n+2}\}})-L_n(\mathbb X_2\backslash{\{X_1,X_{n+2}\}}), \\
\mathbb D_1^{''} L_n&=L_n(\mathbb X_2\backslash{\{X_{2},X_{n+1}\}})-L_n(\mathbb X_2\backslash{\{X_1,X_2\}}).
\end{align*}
Note that both components of $\mathbb D_1^{'} L_n$ are dependent on $X_2$ and are independent of $X_{n+2}$.
For $\mathbb D_1^{''} L_n$ it is conversely. Now by constructing the set $\mathcal R_2$ from the set $\mathcal R_1$ we find the following cases. 

For $R_1<R_{n+1}$,

\begin{equation*}
n\mathbb D_1^{'} L_n=\begin{cases}
-\displaystyle\sum_{j=R_1}^{R_{n+1}-1} c_{j-1}\mathbf{\Delta}_{j:n+2} &\text{for $R_{n+2}<R_1$,}\\
-\displaystyle\sum_{j=R_1}^{R_{n+2}-1} c_j\mathbf{\Delta}_{j:n+2}
-\displaystyle\sum_{j=R_{n+2}}^{R_{n+1}-1} c_{j-1}\mathbf{\Delta}_{j:n+2} &\text{for $R_1<R_{n+2}<R_{n+1}$,}\\
-\displaystyle\sum_{j=R_1}^{R_{n+1}-1} c_j\mathbf{\Delta}_{j:n+2} &\text{for $R_{n+1}<R_{n+2}$,}
\end{cases}
\end{equation*}

\begin{equation*}
n\mathbb D_1^{''} L_n=\begin{cases}
-\displaystyle\sum_{j=R_1}^{R_{n+1}-1} c_{j-1}\mathbf{\Delta}_{j:n+2} &\text{for $R_{2}<R_1$,}\\
-\displaystyle\sum_{j=R_1}^{R_{2}-1} c_j\mathbf{\Delta}_{j:n+2}
-\displaystyle\sum_{j=R_{2}}^{R_{n+1}-1} c_{j-1}\mathbf{\Delta}_{j:n+2} &\text{for $R_1<R_{2}<R_{n+1}$,}\\
-\displaystyle\sum_{j=R_1}^{R_{n+1}-1} c_j\mathbf{\Delta}_{j:n+2} &\text{for $R_{n+1}<R_{2}$.}
\end{cases}
\end{equation*}

For $R_1>R_{n+1}$,

\begin{equation*}
n\mathbb D_1^{'} L_n=\begin{cases}
\displaystyle\sum_{j=R_{n+1}}^{R_{1}-1} c_{j-1}\mathbf{\Delta}_{j:n+2} &\text{for $R_{n+2}<R_{n+1}$,}\\
\displaystyle\sum_{j=R_{n+1}}^{R_{n+2}-1} c_j\mathbf{\Delta}_{j:n+2}
+\displaystyle\sum_{j=R_{n+2}}^{R_{1}-1} c_{j-1}\mathbf{\Delta}_{j:n+2} &\text{for $R_{n+1}<R_{n+2}<R_{1}$,}\\
\displaystyle\sum_{j=R_{n+1}}^{R_{1}-1} c_j\mathbf{\Delta}_{j:n+2} &\text{for $R_{1}<R_{n+2}$,}
\end{cases}
\end{equation*}

\begin{equation*}
n\mathbb D_1^{''} L_n=\begin{cases}
\displaystyle\sum_{j=R_{n+1}}^{R_{1}-1} c_{j-1}\mathbf{\Delta}_{j:n+2} &\text{for $R_{2}<R_{n+1}$,}\\
\displaystyle\sum_{j=R_{n+1}}^{R_{2}-1} c_j\mathbf{\Delta}_{j:n+2}
+\displaystyle\sum_{j=R_{2}}^{R_{1}-1} c_{j-1}\mathbf{\Delta}_{j:n+2} &\text{for $R_{n+1}<R_{2}<R_{1}$,}\\
\displaystyle\sum_{j=R_{n+1}}^{R_{1}-1} c_j\mathbf{\Delta}_{j:n+2} &\text{for $R_{1}<R_{2}$.}
\end{cases}
\end{equation*}
Note that, it suffices to consider only those elements of $\mathcal R_2$ for which we have $R_1>R_2$.

Similarly, write $\mathbb D_3 L_n=\mathbb D_2^{'} L_n-\mathbb D_2^{''} L_n$, where
\begin{align*}
\mathbb D_2^{'} L_n&=L_n(\mathbb X_3\backslash{\{X_{n+1},X_{n+2},X_{n+3}\}})-
                     L_n(\mathbb X_3\backslash{\{X_2,X_{n+1},X_{n+3}\}}) \\
                   &-L_n(\mathbb X_3\backslash{\{X_1,X_{n+2},X_{n+3}\}})+
                     L_n(\mathbb X_3\backslash{\{X_1,X_2,X_{n+3}\}}), \\
\mathbb D_2^{''} L_n&=L_n(\mathbb X_3\backslash{\{X_{3},X_{n+1},X_{n+2}\}})-
                      L_n(\mathbb X_3\backslash{\{X_2,X_{3},X_{n+1}\}}) \\
                    &-L_n(\mathbb X_3\backslash{\{X_1,X_{3},X_{n+2}\}})+
                      L_n(\mathbb X_3\backslash{\{X_1,X_2,X_{3}\}}).
\end{align*}
Note that all the components of $\mathbb D_2^{'} L_n$ are dependent on $X_3$ and are independent of $X_{n+3}$.
For $\mathbb D_2^{''} L_n$ it is conversely. Now, by constructing the set $\mathcal R_3$ from the set $\mathcal R_2$, similarly as in the case $k=2$, we can recursively find $\mathbb D_3 L_n$.
Without loss of generality, we consider only those elements of $\mathcal R_3$ for which $R_1>R_2>R_3$.
The calculations are routine, long and cumbersome; therefore we omit them.

The calculation of $\mathbb D_4 L_n$ is similar. Assuming that $R_1>R_2>R_3>R_4$, it suffices to consider (only) $8!/4!$ elements of the set $\mathcal R_4$.

\end{proof}

\small

\bibliographystyle{plain} 
\bibliography{x}
 
\end{document}